\documentclass[11pt,a4paper,reqno]{amsart} 

\usepackage{color,latexsym,amsfonts,amssymb,amsmath,cite, amsthm}
\usepackage{hyperref}
\usepackage{fullpage}
\usepackage{graphicx}
\usepackage{tikz}
\usepackage{dsfont}
\usepackage{nicefrac,enumerate}

\providecommand{\R}{\mathbb{R}}

\providecommand{\D}{\mathbb{D}}

\renewcommand{\L}{\mathcal{L}}

\providecommand{\Z}{{\mathbb{Z}}}
\providecommand{\Zn}{\Z_{\ge0}}
\providecommand{\Q}{\mathbb{Q}}

\renewcommand{\P}{\mathbb{P}}

\providecommand{\G}{\mathcal{G}}
\providecommand{\Grak}{G^{\ms{rc}}}

\providecommand{\E}{\mathbb{E}}

\providecommand{\C}{\mc{C}}
\renewcommand{\nicefrac}[2]{#1/#2}
\providecommand{\M}{\mathcal{M}}

\providecommand{\T}{\mathbb{T}}

\renewcommand{\d}{{\textup{d}}}
\providecommand{\res}{\mathord{\upharpoonright}}
\providecommand{\deg}{\mathsf{deg}}

\providecommand{\Pt}{\mc{P}_{\theta}}

\providecommand{\es}{\emptyset}
\providecommand{\one}{1}

\providecommand{\mc}{\mathcal}

\providecommand{\Gs}{\mc{G}^*}

\providecommand{\Sd}{D_m}

\providecommand{\g}{\gamma}

\providecommand{\e}{\varepsilon}
\providecommand{\la}{\lambda}

\providecommand{\ms}{\mathsf}

\newcommand{\lgood}{\mathsf{lgood}}

\newcommand{\nf}{\nicefrac}

\newtheorem{theorem}{Theorem}[section]
\newtheorem{corollary}[theorem]{Corollary}
\newtheorem{conjecture}[theorem]{Conjecture}

\newtheorem{lemma}[theorem]{Lemma}
\newtheorem{proposition}[theorem]{Proposition}
\theoremstyle{definition}

\newtheorem{remark}[theorem]{Remark}

\date{\today}

\begin{document}

\author{Christian Hirsch}
\address[Christian Hirsch]{Department of Mathematical Sciences, Aalborg University, Skjernvej 4, 9220 Aalborg \O, Denmark}
\email{christian@math.aau.dk}

\author{Christian M\"onch}
\address[Christian M\"onch]{Fachbereich Mathematik, Technische Universit\"at Darmstadt, Schlossgartenstr. 7, 64289 Darmstadt, Germany.}
\email{moench@mathematik.tu-darmstadt.de}

\title{Distances and large deviations in the spatial preferential attachment model}

\thanks{CH is supported by The Danish Council for Independent Research | Natural Sciences, grant DFF -- 7014-00074 \emph{Statistics for point processes in space and beyond}, and by the \emph{Centre for Stochastic Geometry and Advanced Bioimaging}, funded by grant 8721 from the Villum Foundation.}

\begin{abstract}
We investigate two asymptotic properties of a spatial preferential-attachment model introduced by E.~Jacob and P.~M\"orters ~\cite{jacMor3}. First, in a regime of strong linear reinforcement, we show that typical distances are at most of doubly-logarithmic order. Second, we derive a large deviation principle for the empirical neighbourhood structure and express the rate function as solution to an entropy minimisation problem in the space of stationary marked point processes. 
\end{abstract}

\maketitle
\section{Introduction}
\label{introSec}

Network scientists discovered that many real-world networks appear to be \emph{scale-free} \cite{albert2002statistical}, i.e. their node degrees approximately follow a power law. Consequently, scale-free network models have attracted a lot of attention from researchers in the last two decades. The monograph \cite{van2016random} provides an excellent introduction into the subject from a mathematical point of view. In their landmark paper \cite{BabA99}, Barab\'asi and Albert observed that scale-free networks emerge naturally from simple local rules based on a reinforcement scheme known as \emph{preferential attachment (PA)}. However, being locally tree-like, classical PA networks fail to reflect the clustering effects common in many real-world networks \cite{dereich2013random,berger2014asymptotic}.\\
\begin{figure}[!htpb]
	\input{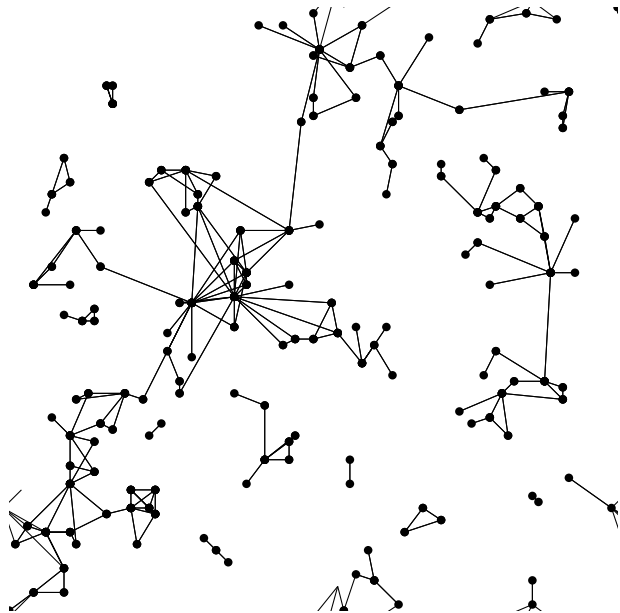}
	\caption{Illustration of the S-PAM from \cite{jacMor1} on the two-dimensional torus.}
	\label{fig}
\end{figure}

An elegant approach to incorporate clustering into PA models is to use geometry. We embed the network nodes into Euclidean or hyperbolic space and make the PA mechanism aware of spatial distances. Thus, nearby nodes are likely to be connected, thereby inducing clustering \cite{jacMor1}. In this sense, \emph{spatial preferential attachment models (S-PAMs)} combine the virtues of PA models and classical geometric random graphs \cite{penrose}, which exhibit strong local clustering but are not scale-free.

The literature offers a variety of definitions and results for spatial PA models \cite{jacMor1, jacMor2, jacMor3, jacMor4, aiello2008spatial,janssen2013geometric,jordan2010degree,jordan2013geometric, flaxman1,flaxman2, pra1,pra2}. In the present paper, we investigate typical distances and large deviation principles (LDPs) in the model introduced in \cite{jacMor3}, as illustrated in Figure \ref{fig}. Our methods are fairly robust and we believe that in essence our analysis could be transferred to many of the other architectures mentioned above. 
\\

Firstly, we verify an intriguing conjecture in \cite[Remark 3]{jacMor2} about typical distances in the largest connected component of the S-PAM. As made precise in Theorem \ref{distanceThm} below, the asymptotic behaviour depends on both the strength of the preferential attachment as well as on the influence of vertex distances on the connection probability. This is in contrast to the situation for the asymptotic degree distribution. Indeed, here only the strength of the preferential attachment but not the geometry affects the power-law exponent \cite[Remark 1]{jacMor1}.

There are a number of small-world results for complex networks endowed with a geometric structure already available \cite{sfPerc,ultraSmall,komj,komj2,flaxman2}. Our work complements these earlier investigations naturally by providing a distance result in a spatial model with preferential attachment in the ultra-small regime of doubly-logarithmic typical distances. Our proof shows that, in this regime, the geometry of the underlying space has an influence on the length of shortest paths in the graph. However, the geometry does not change their fundamental architecture as it can already be observed in non-geometric PA models \cite{DvdH10,DMM12,caravenna2016diameter}.\\
Secondly, we establish an LDP for the in-degree evolution and the evolution of the neighbourhood structure of a typical vertex. The most fundamental building blocks for this result are the process-level LDP of a marked Poisson point process \cite{georgii2} and the contraction principle. The rate function is expressed via a constraint minimisation problem of the specific relative entropy. Loosely speaking, large deviations from the neighbourhood structure are induced by stationary modifications of the original Poisson point process of vertices. Asymptotically, the configurations in a rare event minimise the specific-entropy costs of these modifications subject to the constraint of leading to the considered rare event. Hence, our approach offers a novel complementing perspective to previously used martingale techniques \cite{choi, subLinPA}.

In a broader context, our investigation embeds into the current stream of research on large deviations in random graphs, which has already resulted in a variety of deep and beautiful mathematical results. We refer the reader to \cite{chattSurv} for an excellent survey on the breakthroughs for Erd\H{o}s-R\'enyi graphs. For PA models, large deviations for the degree sequence are available both in the regime of linear as well as sub-linear reinforcement \cite{choi, subLinPA}. In general, in the spatial setting, geometric constraints often give rise to results that are surprisingly different from purely combinatorial models \cite{harel}.\\

The rest of the paper is organised as follows. In Section \ref{modelSec}, we provide precise definitions of the S-PAM and state our two main results. Sections \ref{distSec} and \ref{ldpSec} contain the proofs for the distance asymptotics and the large deviation principle, and in Section \ref{conclusion} we discuss extensions of our results and open problems. Finally, the appendix collects technical auxiliary results that have already appeared in a related form in \cite{jacMor2}.

\newpage
\section{Model definition and main results}
\label{modelSec}

\subsection{Definition}
\label{defSec}
We consider the S-PAM from~\cite{jacMor1, jacMor2, jacMor3, jacMor4}. More precisely, the ambient space $\T_n = [-\nicefrac{n^{1/d}}2,\nicefrac{n^{1/d}}2]^d/{\sim}$ of the model is the torus of side length $n > 0$ in dimension $d\ge1$. For the construction of the S-PAM, we view the process of network nodes as a space-time process of points arriving sequentially in $\T_n$. More precisely, they form a homogeneous Poisson point process $X = X_n$ on $\T_n \times [0, 1]$ with intensity $1$. Formally, a point $(x,s) \in X$ is a \emph{vertex} at \emph{position} $x \in \T_n$ and \emph{birth time} $s \in [0,1]$. For brevity, we often write just $x \in X$ to denote the a.s.~unique vertex $(x,s)$ in position $x\in \T_n$.

 We identify $X$ with the vertex set of a random geometric graph $G_n = (X,E)$ obtained by the following distance-dependent preferential attachment mechanism.

The model is parametrised by an affine function $f:\, \Zn \to (0,\infty)$, $z \mapsto \g z + \g'$ inducing the PA mechanism, where $\g\in(0,1), \g' > 0$, and a decreasing {profile function} $\varphi:\, [0, \infty) \to [0,1]$ incorporating the spatial effects. We assume power decay of the profile function in the sense that $\varphi(x) = \min\{\kappa x^{-\delta}, 1\}$ for some $\delta > 1$ with normalising constant $\kappa$ chosen such that $\int_0^\infty \varphi(x) \d x = 1/2$.
With these settings, the edges in the S-PAM are obtained as follows, where by a slight abuse of notation we write $|\cdot - \cdot|$ for the Euclidean distance on the torus. Initially, the edge set is empty. Whenever a new vertex $(y, t) \in X$ is born, it connects to each vertex $(x, s)\in X$ with $s<t$ independently with probability 
\begin{equation}\label{eq:connprob}
\varphi\Big(\frac{t |x - y|^d}{f(Z_x(t-))}\Big),
\end{equation}
where $Z_x(t-)$ denotes the \emph{in-degree} of $(x,s)$ at time $t-$, i.e.~the number of connections it has already received from vertices born during $(s,t)$. These dynamics give rise to an increasing process $(G_n(t))_{t\in[0,1]}$ of geometric graphs. We usually write $G_n$ for $G_n(1)$.

We conclude our definition of the S-PAM with a few remarks regarding the difference of our setup to the definition given in~\cite{jacMor1,jacMor2}. It is shown in~\cite[Section 4.1]{jacMor1} that the model therein, say $(\tilde G_n)_{n \ge 0}$, is invariant under a particular space-time rescaling, mapping $\tilde G_n$ to our $G_n$. $G_n$ is then used throughout most of the proofs in~\cite{jacMor1,jacMor2}. Since our arguments also use exclusively the rescaled version of their graph, our definition refers directly to the image $G_n$ of $\tilde G_n$ under rescaling. However, it is important to note that the nature of the rescaling is static, i.e.~we first have to build the graph and then rescale it; the rescaling \emph{does not} bijectively map the temporal graph process $(\tilde G_n)_{n \ge 0}$ to the spatial process $(G_n)_{n \ge 1}$ which is obtained in the obvious way from our model by coupling the corresponding Poisson processes.

\subsection{Typical distances}
One of the central findings in \cite{jacMor2} is that the S-PAM undergoes a phase transition: if the attachment function $f$ increases quickly and the profile function $\varphi$ decreases slowly, then the connected component $C_n$ of the oldest vertex in $G_n = G_n(1)$ grows linearly in $n$. Moreover, this component is \emph{robust} under site percolation in the sense that it remains of linear size even after any nontrivial iid Bernoulli thinning of the vertices. If the attachment is weak or the profile decays rapidly, then all connected components grow sub-linearly in $n$.

We focus on the robust phase. More precisely, by \cite[Theorem 1]{jacMor2} and \cite[Theorem 7]{jacMor1}, for $\gamma > \nicefrac\delta{(1+\delta)}$ we have
\begin{equation} \label{eq:giant}
\P-\lim_{n\to\infty}\frac{\#C_n}n=\theta\in(0,1),
\end{equation}
where $\P-\lim$ is shorthand for limit in probability. In this regime it is conjectured \cite[Remark 3]{jacMor2}, that typical distances are of doubly-logarithmic order. We verify this conjecture here. The graph distance in $G_n$ is denoted by $\textup{dist}_n(\cdot,\cdot)$.
 
\begin{theorem}[Distances for $\gamma > \nicefrac\delta{(1+\delta)}$]
    \label{distanceThm}
Let ${Y},{Y'}$ be uniformly chosen vertices of $C_n$. Then, with high probability as $n\to\infty$,
	\[
\mathsf{dist}_n({Y},{Y'}) \le \big(4 + o(1)\big)\frac{\log\log n}{\log \frac\gamma{\delta(1 - \gamma)}}.
\]
\end{theorem}
In the limit $\delta \to 1$, the influence of the geometry of $\T_n$ vanishes and we recover the scaling of typical distances in non-spatial PA models, cf. \cite{DMM12}. The precise value $\log(\gamma/(\delta(1 - \gamma)))$ of the factor in front of $\log\log n$ results from a quick back-of-an-envelop calculation.

As in the standard PA model, vertices born at time $s$ typically have degree $s^{-\gamma}$ for small $s$. In particular, a vertex born at time $n^{-a}$ connects to about $n^{a\gamma}$ vertices, each connecting to the oldest vertex born at time $n^{-1}$ with probability $(n/n^\gamma)^{-\delta} = n^{\delta (\gamma - 1)}$. Here, we find the difference to the standard PAM, where such connections occur with probability of order $n^{\gamma - 1}$. 
Hence, we can expect to find a connection if $a \approx \nf{\delta(1 - \gamma)}{\gamma}$. A more refined spatial analysis reveals that in general vertices born at time $n^{-a^{k+1}}$ connect to vertices born at time $n^{-a^k} = \exp(-\exp( k \log(a)+\log \log n))$ via two edges. Hence, the number of iterations needed to connect a vertex of bounded degree to the oldest vertex should be of the order $-\nf{\log\log n}{\log a}$.

\subsection{Large deviations principle}
\providecommand{\neighb}{\ms{neighb}}

As our second main result, we derive an LDP for the evolution of the neighbourhood structure in the S-PAM as defined in Section~\ref{defSec}, in the vein of~\cite{caputo}. This complements results for combinatorial sparse graphs discussed in~\cite{caputo} by a class of random graphs with an underlying geometry. In particular, as a corollary we obtain an LDP for the evolution of the empirical in-degree distribution over time, similar to the setting in~\cite{binBalls, choi}. 

To state the LDP precisely, we introduce notation related to local convergence of graphs. We let $\Gs$ denote the family of \emph{rooted graphs}, i.e.~of locally finite and connected graphs with a distinguished vertex. We write $g_h$ for the subgraph of a rooted graph $g \in \Gs$ obtained as the union of all paths in $g$ connecting to the root in at most $h\ge0$ hops. We equip $\Gs$ with the \emph{local topology}, the topology generated by the functions $\ms{ev}_{h, g'}:\, g \mapsto \one\{g_h \simeq g'_h\}$, where $h \in \Z_{\ge0}$ and $g' \in \Gs$. Then, writing $[G_n(t), x]$ for the spatial PAM at time $t$ with distinguished vertex $x \in X$, the \emph{evolution of the empirical neighbourhood structure} 
\begin{align}
	\label{empNeighbEq}
	L_n^{\neighb}(\cdot) = \frac1n \sum_{x \in X} \delta_{[G_n(\cdot), x]}
\end{align}
defines a random variable in the product space $\M(\Gs)^{[0, 1]}$, where $\M(\Gs)$ is the family of finite measures on $\Gs$ endowed with the vague topology. In other words, $\M(\Gs)^{[0,1]}$ carries the smallest topology such that for each $t \in [0,1]$, $h \in \Z_{\ge0}$ and rooted graph $g\in\Gs$ the evaluation maps
\begin{align*}
	\ms{ev}_{t, h, g}:\, \M(\Gs)^{[0,1]} &\to [0,\infty) \\
	\nu &\mapsto \nu_t(g_h)
\end{align*}
are continuous. Also note that since we work in a Poisson setting, in~\eqref{empNeighbEq} we normalise by the window size $n$ rather than the random number of vertices.

%
%
By applying the contraction principle~\cite[Theorem 4.2.10]{dz98}, the LDP for $L_n^{\neighb}$ becomes a consequence of the LDP for marked Poisson point processes~\cite[Theorem 3.1]{georgii2}. Hence, we introduce common notation in this setting. First, to realise the independent connections with the probability described in~\eqref{eq:connprob}, we proceed as in the random connection model~\cite{netCom} and introduce a family of auxiliary random variables. More precisely, we augment each vertex $(x, s) \in X$ independently with a collection of iid random variables $\{V_{x, y}\}_{y \in X}$  such that each $V_{x,y} \sim {\bf U}([0,1])$ is uniformly distributed on $[0,1]$. In other words, $X$ becomes an $[0,1]^{\Zn}$-marked Poisson point process.  
As a new vertex $(y, t) \in X$ arrives, it connects to $(x, s)\in X$ if and only if $s < t$  and $V_{x, y}$ is smaller than the threshold given in~\eqref{eq:connprob}.

%
%
In the limit $n \to \infty$ the torus $\T_n$ approaches $\R^d$. Therefore, the limiting objects appearing in the LDP live in $\Pt$, the space of all distributions of stationary $[0, 1]^{\Zn}$-marked point processes on $\R^d$ endowed with the $\tau_{\mc L}$-topology of local convergence. This topology is generated by the evaluations $\ms{ev}_f: \Pt \to [0, \infty)$,  $\Q \mapsto \int_{\C} f(\psi) \Q(\d \psi)$, where $\C$ is the space of configurations in the space $\R^d \times [0,1]^{\Zn}$ that are locally finite in the first component and $f$ is any nonnegative measurable function depending only on the configuration in a bounded domain~\cite{georgii2}. Additionally, $\Q^*$ denotes the unnormalised \emph{Palm version} of a stationary point process $\Q \in \Pt$ \cite[Section 9]{poisBook}. That is, $\Q^*$ is determined by the disintegration identity
$$\int_{\C} f(\psi) \Q^*(\d \psi) = \int_{\C} \int_{[0, 1]^d} f(\theta_x \psi) \psi(\d x) \Q(\d \psi),$$
where $\theta_x: \R^d \to \R^d$, $y \mapsto y - x$ denotes the shift by $x \in \R^d$. Then, for each time $t\in[0,1]$, considering the neighbourhood structure at the origin $o \in \R^d$ under the Palm measure $\Q^*$ yields an element in $\M(\Gs)$. Hence, letting $t$ vary, we associate to $\Q \in \Pt$ the evolution of the neighbourhood structure $\Q^{*, \neighb} \in \M(\Gs)^{[0, 1]}$.

Finally, the rate function in the LDP is expressed in terms of the \emph{specific relative entropy} of stationary marked point processes with respect to the marked Poisson point process. More precisely, writing $\res_{[-n/2, n/2]^d}$ for the restriction to the box $[-n/2, n/2]^d$, for $\Q \in \Pt$ we put
$$H(\Q) = \lim_{n \to \infty} n^{-d} \int \log\Big(\frac{\d \Q{\res_{[-n/2, n/2]^d}}}{\d \ms{Pois}\res_{[-n/2, n/2]^d}}(\psi)\Big) \d \Q\res_{[-n/2, n/2]^d}(\d \psi),$$
tacitly  applying the convention that $H(\Q) = \infty$ if the Radon-Nikodym derivative of the restricted point processes does not exist. By means of the contraction principle, the LDP for marked Poisson point processes~\cite[Theorem 3.1]{georgii2} now gives rise to the LDP for the evolution of the empirical neighbourhood structure.

\begin{theorem}
    \label{ldp2Thm}
    The empirical neighbourhood structure $\{L_n^{\neighb}\}_{n \ge 1}$ satisfies the LDP in the product space $\M(\Gs)^{[0, 1]}$ with good rate function
    $$\nu \mapsto \inf_{\substack{\Q \in \mc P_\theta \\ \Q^{*, \neighb} = \nu}} H(\Q).$$
\end{theorem}

Since the degree of the root in a rooted graph is nothing more than the size of the 1-neighbourhood, after another application of the contraction principle, Theorem~\ref{ldp2Thm} yields an LDP for the \emph{evolution of empirical in-degrees} 
\begin{align}
    \label{degEvEq}
    L_n^{\mathsf{deg}}(\cdot) = \frac1n \sum_{k \ge 0}\#\{x \in X:\, Z_x(\cdot) = k\} \delta_k.
\end{align}

\begin{corollary}
    \label{ldpThm}
    The empirical in-degree evolution $\{L_n^{\mathsf{deg}}\}_{n \ge 1}$ satisfies the LDP in the product space $\M(\Zn)^{[0, 1]}$ with good rate function
    $$\nu \mapsto \inf_{\substack{\Q \in \mc P_\theta \\ \Q^{*, \mathsf{deg}} = \nu}} H(\Q).$$
\end{corollary}

Since $\M(\Zn)^{[0, 1]}$ carries the product topology, Corollary~\ref{ldpThm} only provides access to crude information on the time evolution of the empirical degree distributions. For this reason, we next deduce a more refined LDP based on the Skorohod topology \cite{FeKu06}. Since this topology requires an underlying metric space, we consider only the setting of a priori bounded in-degrees. Hence, we replace the $\M(\Zn)$ by a suitable Euclidean space. More precisely, for $k \ge 0$ let
$$L_n^{\mathsf{deg}; \le k}(\cdot) = \Big(\frac1n \#\{x\in X:\, Z_x(\cdot) = 0\}, \ldots, \frac1n \#\{x\in X:\, Z_x(\cdot) = k\}\Big)$$
denote the evolution of the $(k+1)$-dimensional vector containing the normalised in-degree evolutions truncated at the $k$th in-degree. We consider $L_n^{\mathsf{deg}; \le k}$ as a random element of the Skorohod space $\D_{k+1}$ of functions $f:\,[0, 1] \to [0, \infty)^{k+1}$ that are c\`adl\`ag in each coordinate.

\begin{corollary}
    \label{ldpCor}
	For every $k \ge 0$ the truncated empirical in-degree evolution $\{L_n^{\mathsf{deg}}\}_{n \ge 1}$ satisfies the LDP in the Skorohod topology with good rate function
  $$\mathbf f = (f_0(\cdot), \ldots, f_ k(\cdot)) \mapsto \inf_{\substack{\Q \in \mc P_\theta \\ \Q^{*, \deg; \le k} = \mathbf f}} H(\Q).$$
\end{corollary}

\section{Proof of Theorem \ref{distanceThm}}
\label{distSec}

We prove Theorem \ref{distanceThm} in several steps. First, in Section \ref{sec:global} we explain the overall idea and give the proof subject to intermediate results. Then, Section \ref{sec:rob} introduces sprinkling and monotonicity as central tools for the arguments in the subsequent sections. Finally, Sections \ref{sec:loc} and \ref{sec:layersB} contain the proofs of the intermediate results.
\subsection{The main argument}\label{sec:global}
To establish the upper bound on typical distances, we first show that almost all vertices in the giant component are within bounded distance of a fairly old vertex of high degree. Then, we proceed to show that each such high-degree node is at distance at most $2 (\rho + o(1)) \log\log n$ of the oldest vertex in $G_n$ with high probability, where
\begin{equation}\label{def:prefactor}
	\rho = \frac1{\log(\nf\gamma{(\delta(1 - \gamma))}}.
\end{equation}
In essence, this argument is already outlined in \cite{jacMor2}, cf.~Remark 3, and can be traced in the proofs of Propositions 13 and 15 therein, see also the brief heuristics given after the statement of Theorem~\ref{distanceThm}.
However, the arguments given in \cite{jacMor2} to establish the existence of a giant component only require the oldest vertex to connect to sufficiently many lower degree vertices. To show this, only a bounded number of search steps are necessary. To prove Theorem \ref{distanceThm} along similar lines, we analyse the probabilities of adverse events occurring during the search for connecting vertices more thoroughly than is required for the robustness results in \cite{jacMor2}. To keep the different stages of our search algorithm sufficiently independent we rely on a sprinkling construction in the vein of \cite{jacMor2}. More precisely, for some small $r>0$ we colour each vertex in $X$ independently \emph{red} with probability $r$ and \emph{black} with probability $b=1-r$. Then, $G^r_n$ and $G^b_n$ denote the S-PAMs constructed on the red and black vertices, respectively.
The reasoning behind this will be explained in Section~\ref{sec:rob}.\\
\begin{remark}
As is made precise in Lemma \ref{lem:monotonicity} below, the S-PAM satisfies a strong super-additivity principle induced by the reinforcement effect of PA. Let $X_1$ and $X_2$ denote two Poisson processes on $\T_n \times [0,1]$ of positive intensity and let $G(X)$ denote the S-PAM built from $X$. If both profile and PA rule are monotone, then the construction of both S-PAMs can be coupled such that almost surely
\[G(X_1)\cup G(X_2) \subset G(X_1 + X_2).\]
\end{remark}

Let us make the overall argument precise. To start the construction, we need to find an old black vertex near a uniformly chosen vertex $Y \in C_n$. By stationarity, we may assume that $Y = (o,U)$ is located at the origin $o \in \R^d$ with $U$ uniform in $[0,1]$ and consider the Poisson process $X$ under the corresponding Palm distribution $\P_{(o,U)}$. A vertex $(x,s) \in G^b_n$ is \emph{$D$-reachable} if it connects to $(o, U)$ by a path in $G^b_n$ in at most $D$ hops.  For ease of reference, we introduce the events
\[E^b_n(D, s) = \{\text{some vertex $Y_0\in G^b_n$ born before time $s$ is $D$-reachable} \}.\]
If there are several reachable vertices, $Y_0$ denotes the one with minimal birth time.

\begin{proposition}[Connection to good vertices]
	\label{prop:local}
	Let $b, s > 0$. Then, there exists an almost surely finite random variable $D = D^b(s)$
\[
	\lim_{n \to \infty}\P_{(o,U)}\big(\{(o,U) \in C^b_n\}\, \setminus\, E^b_n(D^b(s), s)\big) = 0,
\]
where $C^b_n$ denotes the connected component of the oldest vertex in $G^b_n$.
\end{proposition}

Once we have reached a sufficiently old black vertex, we proceed as in the above heuristic argument. 
For the remainder of this section, $g: (0, \infty) \to (0, \infty)$ denotes the sub-polynomially growing function introduced in Lemma \ref{lem:oldgood}, which is parametrised by $\gamma,\delta$ and $r$ only. 
Extending a notion from \cite{jacMor2}, we say a vertex $(x,s) \in G_n$ is \emph{$r$-good} if $s < \nf12$ and it has at least $\nf{s^{-\gamma}}{g(s^{-1})}$ red neighbours with birth times in $(s,\nf{1}{2})$. It is \emph{locally $r$-good} if it remains $r$-good after removing all edges of the form $y \to x$ with $y \notin [x - s^{-\nf1d}, x + s^{-\nf1d}]^d.$ 
Loosely speaking, exploring possible paths along good vertices offers the advantage that we have a sufficient number of outgoing connections to choose from. Additionally, local goodness allows us to scan $X$ for good vertices while keeping the explored areas sufficiently localised to leverage on the spatial independence of Poisson points.

We build up a hierarchical connection path along $r$-good vertices of increasing age joined by young red vertices born after time $\nicefrac12$. Writing $\mathsf{rgood}_n \subset X$ for the subset of all red $r$-good vertices, we introduce a hierarchy of layers 
\[L^r_1 \subset L^r_2 \subset \dots \subset \mathsf{rgood}_n\] 
of red $r$-good vertices, parametrised by their age. The first layer $L^r_1$ contains the vertices of highest degree, i.e.~near $n^\gamma$. With increasing index, the layers $\{L^r_i\}_{i \ge 1}$ contain more and more vertices of lower and lower degrees. More precisely, in the robust regime $\gamma > \nf{\delta}{(1 + \delta)}$, we can fix global parameters
\begin{align}
	\label{abEq}
\alpha \in\Big(1,\frac\gamma{\delta(1 - \gamma)}\Big), \; \beta \in \Big(\alpha, \frac\gamma\delta + \alpha\gamma\Big),
\end{align}
and then set 
\[L^r_k = \left\{(x,s) \in\mathsf{rgood}_n:\, s \le n^{-{\alpha^{-k}}}\right\} \]
and \[
	K = \min\left\{k \ge 1:\, n^{-{\alpha^{-k}}} \le (\log n)^{-\nu^{-1}}\right\}-1,
\]
where \[
\nu = \min\left\{-\beta\delta+\gamma-\alpha\gamma\delta,\frac{\beta-\alpha}{d}\right\}>0.
\]
Starting from an old $r$-good vertex, we typically reach $L^r_K$ in at most $C(\alpha,\beta,r)\log\log\log n$ steps, where $C(\alpha,\beta,r)$ is a sufficiently large constant. In particular, in Section \ref{sec:layersB}, we explicitly specify the scheme for establishing these connections. For the moment, assume that $C(\alpha,\beta,r)$ is given and call an $r$-good vertex at distance at most $C(\alpha,\beta,r)\log\log\log n$ from $L^r_K$ \emph{well-connected}.
\begin{proposition}[Well-connectedness]
	\label{prop:layersB}
	Let $b > 0$. Then,
	\[\lim_{s \to 0} \liminf_{n \to \infty}\E[\P_{Y_0}\big(Y_0 \text{ is well-connected}\,\big|\, G^b_n \big)\one\{ E^b_n( D^b(s), s)\}] = 1.\]
\end{proposition}

Having established a path from $(o,U)$ to $L^r_K$, the last step is to bound the diameter of $L^r_K$.
\begin{proposition}[Final layer diameter]
	\label{prop:layersC}
	With high probability, we have that 
	$$\mathsf{diam}_n(L^r_K) \le 4 K.$$
 \end{proposition}
Combining Propositions \ref{prop:local}--\ref{prop:layersC}, we now complete the proof of Theorem \ref{distanceThm}.
\begin{proof}[Proof of Theorem \ref{distanceThm}]
	Since $K$ is of order $(1 + o(1))\log \log n / \log \alpha$, by Proposition \ref{prop:layersC} it suffices to show that a uniformly chosen $Y \in C_n$ connects to $L^r_K$ in at most $o(\log\log n)$ hops with high probability. In particular, this occurs under the event $E \cap F$ where $E = E^b_n(D^b(s), s)$ and $F = \{Y_0 \text{ is well-connected}\}$. In other words, it suffices to show that
	\begin{align}
		\label{cefEq}
		\lim_{b \to 1} \liminf_{s \to 0} \liminf_{n \to \infty} \P_{(o, U)}(\{(o, U) \in C_n^b\} \cap E \cap F) = \theta .
	\end{align}
	To achieve this goal, decompose the left-hand side as
	\begin{align*}
		&\P_{(o, U)}(\{(o, U) \in C_n^b\} \cap E \cap F) \\
		&\qquad\ge \P_{(o, U)}((o, U) \in C_n^b) - \P_{(o, U)}(\{(o, U) \in C_n^b\} \setminus E) - \P_{(o, U)}(E \setminus F).
	\end{align*}
By Proposition \ref{prop:local}, the second summand tends to 0 as $n \to \infty$.
	Since $E$ is measurable with respect to $G^b_n$, the third contribution equals
	$$\P_{(o, U)}(E \setminus F) = \E_{(o, U)}[(1 - \P_{Y_0}(F|G^b_n))\one\{E\}],$$
	which tends to 0 as $s \to \infty$ and $n \to \infty$ by Proposition \ref{prop:layersB}.  Hence, \eqref{cefEq} gives that
	\begin{align*}
		\liminf_{s \to \infty} \liminf_{n \to \infty}\P_{(o, U)}(\{(o, U) \in C_n^b\} \cap E \cap F) \ge \theta^b,
	\end{align*}
	which by continuity of the percolation probability \cite[Proposition 7]{jacMor2} tends to $\theta$ as $b \to 1$. 
\end{proof}
\subsection{Sprinkling and monotonicity}
\label{sec:rob}
In this subsection, we highlight sprinkling and monotonicity as central tools entering the proofs of Propositions \ref{prop:layersB} and \ref{prop:layersC} and in particular we explain which role the colouring plays in the proofs. To explain the idea behind sprinkling, assume we explore two disjoint subgraphs of $G_n$. If we want to show that these subgraphs are connected to each other, we would like to use that they are independent. However, it is sometimes challenging to exclude hidden dependencies in the construction or definition of the subgraphs in question. It would be much easier to sample edges independently between them. This is where the sprinkling technique enters the stage.

Since $G_n$ is built from a Poisson point process, it is easiest to add a few additional points to $X$, which potentially results in additional edges in the S-PAM. It is convenient here to consider again the alternative formulation of the model obtained by first considering the set $X \times X$ of `potential edges' and then assigning the collection of iid weights $V_{X \times X} = \{V_{x, y}\}_{x,y \in X}$ to the potential edges, with $V_{x, y}$ uniform on $[0,1]$ and an edge is added between $(x,s)$ and $(y,t)$ with $s<t$ if and only if $V_{x,y} \le \varphi(t|x-y|^d/f(Z_x(t-)))$.

Formally, we consider, for $b$ close to 1, the independent colouring of the Poisson process $X$ described in the previous section, i.e. each node is either \emph{black} with probability $b$ or \emph{red} with probability $r = 1 - b$. Hence, the thinning theorem for Poisson processes \cite[Corollary 5.9]{poisBook} decomposes $X = X^b \cup X^r$ into a black and an independent red Poisson process with parameters $b$ and $r$, respectively. Recall that $G^b_n$ denotes the S-PA graph built from $X^b$ only, which can be viewed as a site-percolated version of $G_n$ with retention parameter $b$. 

\begin{remark}\label{percneqperc}
The robustness results of \cite{jacMor2} are formulated for this version of site percolation. However, the standard meaning of `site percolation' on $G_n$ requires to \emph{first} construct $G_n$ and \emph{then} remove vertices and incident edges independently with retention probability $b$. However, denoting this classical variant of the percolated graph by $G_{n,b}$, it is easily seen from the monotonicity of the attachment mechanism that we can couple $G_{n,b}$ and $G^b_n$ in such a way that
$
G^b_n \subset G_{n,b},
$
and such that in particular the vertex sets of both graphs coincide.
\end{remark}

By continuity, the S-PAM $G^b_n$ obtained from $X^b$ resembles $G_n$, for $b$ close to 1 and we view the edges sent from a red to a black vertex as a version of sprinkling. 
Now, the following monotonicity principle holds, where for two geometric graphs $G,H$ we write $G \subset H$ if every vertex and every edge of $G$ is also contained in $H$.
\begin{lemma}[Monotonicity]\label{lem:monotonicity}
The graphs $G_n$, $G^b_n$ and $G^r_n$ can be defined on the same probability space in such a way that almost surely
$
G^b_n \cup G^r_n\subset G_n.
$
\end{lemma}
\begin{proof}
	First, represent the Poisson process as $X = X^b \cup X^r$, 
	where the latter is an independent superposition of the black Poisson process and  the red Poisson process.
	We sample the edge variables in a consistent manner with the above decomposition. That is,
$ V_{X^b \times X^b} \subset V_{X \times X}$.
Now, we couple $G_n^b, G_n$ via the sequential PA construction. By monotonicity of the attachment and a suitable coupling the corresponding in-degree evolutions, 
$
G_n^b(t) \subset G_n(t)
$
holds for all $t \le 1$.
	It only remains to note that we can also construct $G^r_n$ in a consistent manner such that $G^r_n\subset G_n$. To achieve this, we use an identical copy of $X^r$ and the corresponding restricted weights and just run the construction of $G^r_n$ alongside the construction of $G^b_n, G_n$ above and observe that by monotonicity of the attachment rule, any edge drawn in $G^r_n$ is also drawn in $G_n$. 
\end{proof}

\subsection{Proof of Proposition \ref{prop:local}}\label{sec:loc}
Say that a vertex $(x, s) \in G^b_n$ is $(D,R)$-\emph{reachable} for some $D \ge1$ and $R>0$ if in $G^b_n$ there exists a path $\pi = \big((o,U), (x_1, s_1), (x_2, s_2), \dots, (x,s) \big)$ with vertex positions in $(-R,R)^d$ that connects $(0, U)$ to $(x,s)$ and is of length at most $D$.
We prove a slightly stronger assertion based on the modified reachability events 
\[
E_n^*(R,D,N,s) = \{\text{at least $N$ black vertices born before time $s$ are }(R,D)\text{-reachable} \}.
\]

\begin{proof}[Proof of Proposition \ref{prop:local}]
	In \cite{jacMor2}, it is shown that as $n \to \infty$, the finite graphs $G^b_n$ converge weakly to a local limit graph $H_\infty$ and $\theta^b \in(0,1)$ in \eqref{eq:giant} becomes the proportion of vertices contained in the unique infinite component $K_\infty$ of $H_\infty$. In other words, the probability that a typical vertex is contained in $K_\infty$, but not in connected component of the oldest vertex in $G^b_n$ tends to 0.
	
	Let us consider this limit graph and introduce the event $E^*_\infty(R,D,N,s)$ corresponding to $E^*_n(R,D,N,s)$ in $H_\infty$. By definition, $E^*_\infty(R,D,N,s)$ is a local event and thus the claim of the proposition follows from local weak convergence, if for any $\varepsilon>0$ there exist sufficiently large values $R = R(s),D = D(s)$ such that
	\begin{equation*}
	\P^\infty_{(o,U)}\big(\{(o,U) \in K_\infty\} \setminus E^*_\infty(R,D,N,s)\big)<\varepsilon.
	\end{equation*}
	
	Since we assume $(o,U) \in K_ \infty$, we already know that there exist (shortest) paths connecting $(o,U)$ to at least $N$ black vertices born before time $s$, by increasing $R$ and $D$ further still and using that $H_\infty$ is locally finite, we can reduce the probability that these finitely many paths are longer than $D$ or not contained in $(-R,R)^d$ to some arbitrarily small value $\varepsilon$.
	
\end{proof}

\subsection{Proof of Propositions \ref{prop:layersB} and \ref{prop:layersC}}\label{sec:layersB}
In what follows, we again apply Lemma \ref{lem:monotonicity} to embed the {red graph} $G^r_n$ into $G_n$. We show that high-degree nodes are connected in two steps. Firstly, Lemma \ref{lem:2conn1} states that two moderately old red vertices of high degree are likely to both connect to a {young red} vertex. Hence, they are at graph distance at most $2$ from each other in $G^r_n$, as long as they are sufficiently close in $\T_n$. Secondly, by Lemma \ref{lem:gooddense}, red high-degree vertices are well spread out such that a red high-degree vertex has a red vertex with a much higher degree not too far away in $\T_n$. This is reminiscent of the robustness proof in \cite{jacMor2}. Nevertheless, due to the different nature of our goal, we conduct a more refined analysis. 

For vertices $(x,s), (y,t) \in X^r$ set
\[
	\Psi^r(x,s) = Z_x^r(\nf12) s^{(\beta - \alpha \gamma)\delta},
\]
where $Z^r_{\cdot}(\cdot)$ denotes in-degree evolutions in $G^r_n$. 
The following lemma is a variant of \cite[Lemma 11]{jacMor2} and follows from the more general Lemma \ref{APP1} in the appendix. Here, we say $x$ and $y$ are \emph{$2$-connected} if $x \leftarrow z \rightarrow y \text{ in $G^r_n$ for some $(z,r) \in X^r\res_{\T_n \times [\nf12, 1]}$}$.
\begin{lemma}[$2$-connections]\label{lem:2conn1}
	There exists a constant $c>0$ such that for every sufficiently large $n$ and every locally $r$-good $(x, s), (y,t) \in X^r$ with $s, t \le 1/4$, $Z_x^r(\nf12)^\alpha \le Z_y^r(\nf12)$ and $|x - y|^d \le s^{-\beta}$ we have
\[
	\P\big(\text{$x$ and $y$ are \emph{$2$-connected}} \,\big|\, X^r \cap (\T_n \times [0, 1/2]) \big) \ge 1 - \textup e^{-c r \Psi^r(x, s)},
\]
\end{lemma}

Proceeding as in \cite[Proposition 13]{jacMor2}, we now discover locally good vertices in $X^r$. Since in this section, we always explore $X$ by moving 'towards the right', i.e.~by increasing the first space-coordinate and since we determine local goodness of a vertex $(x,s)$ by peeking into a cube of volume $\nf1s$ around $x$, the following $\sigma$-algebra naturally captures the information collected during the exploration process in $G^r_n$.
\[
	\mc F(x^-) = \sigma\big(X'(x), V\res_{X'(x)}, \text{ with }X'(x) = X^r\res_{[0, x] \times \mathbb R^{d-1} \times [0,1]} \cup X^r\res_{[x - s^{-\nf1d}, x + s^{-\nf1d}]^d \times [0, \nf12]}\big).
\]
We write $\lgood^r_n$ for the family of locally good vertices in $G^r_n$ and also recall from \eqref{abEq} that the connection scheme between the high-degree vertices relies on the two parameters $\alpha,\beta$.

\begin{lemma}[Density of high degrees]\label{lem:gooddense}
	Let $r >0$ be arbitrary. Then, there exists a constant $q\in(0,1)$ with the following property. If $(x,s) \in G^r_n$ is any vertex with $n^{-\nf1\beta}<s \le 1/4$, then
\[
	\P\big(\lgood^r_n \cap (B_{s^{-\beta /d}}(x) \times [0, s^\alpha]) \ne \es\,|\,\mc F(x^-)\big) \ge 1 - q^{\frac17{s^{-\nf{(\beta-\alpha)}d}}}.
\]
\end{lemma}
\begin{proof}
 For $\eta = \beta - \alpha$ we select $M = \left\lfloor s^{-\nf\eta d}/6 -1\right\rfloor$ disjoint sub-intervals $I_1, \ldots, I_M$ with midpoints $a_1, \ldots, a_M$ within the interval $[s^{-\nf\alpha d}, s^{-\nf\beta d}]$. Now, define $d$-dimensional blocks
\[
A_k = (a_k,0,\dots,0)+[-3s^{-\nf\alpha d},3s^{-\nf\alpha d}]^d,
\]
and \[
B_k = (a_k,0,\dots,0)+[-s^{-\nf\alpha d},s^{-\nf\alpha d}]^d.
\]
The blocks $x+A_k$ are disjoint and their point configuration is independent of $\mc F(x^-)$. Next, note that the total number of points in $(x + B_k) \times (\nf{s^\alpha}2, s^\alpha)$ is Poisson distributed with parameter of constant order proportional to $r$. Moreover, by Corollary \ref{lem:localgood}, the expected number among them that fail to be locally $r$-good is bounded away from $0$. Hence, each of the blocks $x + B_k$ contains a locally good vertex born in $(\nf{s^\alpha}2, s^\alpha)$ with probability at least $1 - q\in(0,1)$.
	
	Now, to check local goodness we need to check in the worst case a set of diameter $2^{\nf1d}s^{-\nf\alpha d} \le 2s^{-\nf\alpha d}$, i.e.~these vertices occur independently for different blocks. Hence, the number of locally good vertices at distance at most $s^{-\nf\beta d}$ from $x$ dominates a multinomial random variable with $M$ trials and success probability $1 - q$, which exceeds $0$ with probability $1 - q^M$.
\end{proof}
We are now in the position to prove the main result of this section, Proposition \ref{prop:layersB}.
Since, $Z^r_y(\nf12) \ge Z^r_x(\nf12)^\alpha$ holds for any $x \in L^r_{k+1}\setminus L^r_k$ and $y \in L^r_k$, by Lemma \ref{lem:2conn1}, with high probability, a vertex in $L^r_{k + 1}$ is $2$-connected to a vertex in $L^r_k$.\\ 
\begin{proof}[{Proof of Proposition \ref{prop:layersB}}]
	Let $Y_0 = (x_0, t_0)$ denote the vertex guaranteed by the event $E^b_n(s)$. 
	We wish to apply first Lemma \ref{lem:gooddense} and then Lemma \ref{lem:2conn1} to find a locally good red vertex $(x_1,t_1)$ with $|x_0 - x_1|^d \le t_0^{-\beta}$ and $t_1 \le t_0^\alpha$ that $2$-connects to $(x_0, t_0)$, thus establishing that $(x_0, t_0)$ and $(x_1,t_1)$ are at distance at most $2$ in $G^r_n$. The probability that this fails is bounded by
\begin{equation}\label{eq:firststep}
	e_1 = q^{{\frac17{t_0^{-\nf{(\beta-\alpha)}d}}}} + \exp(-c r \Psi^r(x_0, t_0)).
\end{equation}
Iteration yields 
\begin{equation*}
	e_j = q^{{\frac17{t_{j-1}^{-\nf{(\beta-\alpha)}d}}}} + \exp(-c r \Psi^r(x_{j-1}, t_{j-1})).
\end{equation*}
	Note that $\Psi(x_{j-1}, t_{j-1}) \ge t_{j-1}^{\beta\delta-\gamma-\alpha\gamma\delta}/g(t_{j-1}^{-1})$ and we recall that 
	\[\nu = \min\{-\beta\delta+\gamma+\alpha\gamma\delta,\nf{(\beta-\alpha)}d\}>0.\]
	Hence, we can find a small number $q>0$ with
\begin{equation*}
	e_j \le 2\exp(-q t_{j-1}^{-\nu}) \le 2\exp(-q t_0^{-\nu \alpha^j}).
\end{equation*}
	The probability of failing to reach $L^r_K$ from $(x_0, t_0)$ in $G^r_n$ is thus bounded by 
	\[2\sum_{j \ge 1} \exp(-q t_0^{-\nu \alpha^j}) \le 2\sum_{j \ge 1} \exp(-q s^{-\nu \alpha^j}),\]
	which can be made arbitrarily small by lowering $s$, see Lemma \ref{lem:vanish}. Note that it takes at most $O(\log\log\log n)$ iterations to arrive at a vertex with birth time $\nf1{(\log n)^C}$ for any $C>0$, since the birth time of the freshly discovered vertex is lower by at least a fixed power than the birth time of the last vertex in each iteration. 
\end{proof}
For Proposition \ref{prop:layersC} we proceed similarly as in Proposition \ref{prop:layersB}.
	\begin{proof}[{Proof of Proposition \ref{prop:layersC}}]
Let a vertex in $(x,s) \in L^r_K$ be given. We need to consider two cases: $s<n^{-\nf1\beta}$ and $s \in(n^{-\nf1\beta},n^{-\nf1{\alpha^k}}).$ In the first case, we argue directly as in the proof of \cite[Proposition 15]{jacMor2} to obtain that $(x,s)$ is either the oldest vertex in $X^r$ or, by Lemma \ref{APP1}, it connects to it with probability exceeding $1 - \e^{\log s^2}$. Note that there are at most $O(n^{1-\nf1\beta})$ such vertices, i.e.~this argument holds with high probability simultaneously for all of them.\\

Let us now consider $s \in(n^{-\nf1\beta},n^{-\nf1{\alpha^k}}).$ An iteration as in the proof of Proposition \ref{prop:layersB} yields a chain of $2$-connections connecting $(x,s)$ to the oldest vertex in at most $K$ steps. Since the error bound is weakest in the first step, we may bound the total probability that the desired path does not exist by \[
	\bar{q}(n) = K\exp({-q {((\log n)^{-\nf\alpha{\nu}}})^\nu }).
\]
Since $\alpha>1$ we have $n \bar q(n) = o(1)$ and thus the total probability that we fail to connect \emph{any} node in $L^r_K$ to the oldest vertex in $K$ steps vanishes, because with high probability there are at most of order $n$ vertices in the system. Thus, the diameter of $L^r_K$ is at most $4 K$ in $G_n^r$.
		
\end{proof}
\begin{remark}
In fact, the diameter result is completely independent of the colouring procedure and the architecture of the connections between the layers remains intact for any percolated version of the graph, which illustrates the robustness of the network. In fact the presence of this hierarchical structure which is build on $o(n)$ nodes, but is still contained in practically all shortest paths is the reason why the giant component remains intact under percolation in the first place. Conversely, it is well known that preferential attachment networks are very susceptible to targeted attacks against the hub-hierarchy \cite{eckhoff2014vulnerability}.
\end{remark}

\section{Proof of Theorem \ref{ldp2Thm}}
\label{ldpSec}

In this section, we prove the LDP asserted in Theorem \ref{ldp2Thm} and deduce Corollary \ref{ldpCor} by applying the LDP for the empirical field of a marked Poisson point process \cite[Theorem 3.1]{georgii2}. For this purpose, we first introduce an approximated network dynamic, where connections appear only up to a finite distance. In a second step, we show that this modified dynamic forms an exponentially good approximation in the sense of \cite[Definition 4.2.14]{dz98}.

%
%
In order to prove Theorem \ref{ldp2Thm}, we rely on the LDP for the empirical field of a marked Poisson point process in the $\tau_{\L}$-topology of local convergence \cite[Theorem 3.1]{georgii2}. However, this result is not directly applicable in the present setting. Indeed, the $\tau_{\L}$-topology captures only interactions of bounded range, whereas the polynomial decay of the profile function $\varphi$ allows for arbitrarily long edges. 

The proof of Theorem \ref{ldp2Thm}, proceeds in two steps. First, in Proposition \ref{ldpApproxProp}, we see that after truncating edges longer than a fixed distance, the resulting neighbourhood evolution is continuous in the input data. In particular, the contraction principle yields an LDP in the truncated setting. Second, Proposition \ref{tvProp} shows that changes induced by the truncations are asymptotically negligible in the sense of exponentially good approximations \cite[Definition 4.2.14]{dz98}. Since the sprinkling construction does not appear in this section, we overwrite the previous notation $G^r$. This approximation has the advantage of exhibiting only local dependencies.

The truncated S-PAM $G^r$ suppresses potential connections longer than a fixed threshold $r > 0$. That is, we consider the dynamics as described in \eqref{eq:connprob}, except that $|x - y|$ is replaced by $\infty$ if $|x - y| > r$.  For $t\le1$ and $\Q \in \Pt$, we write $\Q^{*,r-\neighb}(t)$ for the measure on $\G^*$ determined by the rooted graph $[G^r(t), o]$ under the Palm measure $\Q^*$. Moreover, $\Q^{*,r-\neighb}({t,h})$ denotes the projection of this measure under the map of taking the $h$-neighbourhood.

%
%
\begin{proposition}[LDP for finite-range model]
    \label{ldpApproxProp}
    The approximated neighbourhood evolution $\Q^{*,r-\neighb} \in \mc{M}(\Gs)^{[0, 1]}$ is continuous in $\Q \in \Pt$ under the $\tau_{\L}$-topology.
\end{proposition}
\begin{proof}
	Fix $t \in [0,1]$, $h \ge0$ and $g \in \Gs$. Then, the indicator of the event that $\{[G^r(t), o]_h \simeq g_h\}$ is a local observable. Hence, $\Q^{*,r-\neighb}(t,h)(g_h)$ is continuous in $\Q$ under the $\tau_{\L}$-topology, as asserted.
\end{proof}

Since the empirical field induced by the marked Poisson point process satisfies the LDP in the $\tau_{\L}$-topology with specific entropy as good rate function \cite[Theorem 3.1]{georgii2}, combining Proposition \ref{ldpApproxProp} with the contraction principle implies that $L^{r-\neighb}_n$ satisfies the LDP with good rate function
$$\nu \mapsto \inf_{\substack{\Q \in \Pt \\ \Q^{*,r-\neighb} = \nu}}H(\Q).$$
In order to bridge the gap between $L_n^{r-\neighb}$ and $L_n^{\neighb}$, we rely on the machinery of exponentially good approximation \cite{dz98,eiSchm}.

%
%
\begin{proposition}[Exponentially good approximation]
    \label{tvProp}
Let $t \in [0, 1]$, $h \ge 0$ and $g \in \Gs$. Then, the random variables $L^{r-\neighb}_n(t,h)(g_h)$ are an exponentially good approximation of $L_n(t,h)(g_h)$.
\end{proposition}

Before establishing Proposition \ref{tvProp}, we explain how it enters the proof of Theorem \ref{ldp2Thm}.

%
%
\begin{proof}[Proof of Theorem \ref{ldp2Thm}]
Let $t \in [0,1]$, $h \ge0$ and $g \in \Gs$. First, by \cite[Corollary 1.11]{eiSchm}, it suffices to prove for every $\alpha > 0$ that
    $$\lim_{r \to \infty}\sup_{\substack{\Q \in \Pt\\ H(\Q)\le \alpha}}\big|\Q^{*,r-\neighb}(t,h)(g_h) - \Q^{*,\neighb}(t,h)(g_h)\big| = 0.$$
      Now, for different values of $r$ the approximations $G^r$ are coupled in the sense that for $r' \ge r$ both $\Q^{*,r-\neighb}(t,h)(g_h)$ and $\Q^{*,r'-\neighb}(t,h)(g_h)$ integrate suitable indicators with respect to $\Q^*$. In particular, by the dominated convergence theorem,
    $$\lim_{r \to \infty} \Q^{*,r-\neighb}(t,h)(g_h) = \Q^{*,\neighb}(t,h)(g_h).$$
    Hence, it suffices to show that for $\varepsilon > 0$ there exists $r_0 = r_0(\varepsilon)$ with the following property. If $\Q \in \Pt$ satisfies $H(\Q) \le \alpha$, then 
    \begin{align}
        \label{r0PropEq}
        \sup_{r' \ge r \ge r_0} \big|\Q^{*,r-\neighb}(t,h)(g_h) - \Q^{*,r'-\neighb}(t,h)(g_h)\big| \le \varepsilon.
    \end{align}
    By Proposition \ref{tvProp}, there exists $r_0 > 0$ such that 
    \begin{align}
        \label{expAppEq}
        \limsup_{n \to \infty} \frac1n \log\P\big(\big|L_n^{r-\neighb}(t,h)(g_h) - L_n^{r'-\neighb}(t,h)(g_h)\big| > \varepsilon\big) < -\alpha
    \end{align}
    holds for every $r' \ge r \ge r_0$. As can be deduced from Proposition \ref{ldpApproxProp}, not only $L_n^{r-\neighb}(t,h)(g_h)$ but also the difference $L_n^{r-\neighb}(t,h)(g_h) - L_n^{r'-\neighb}(t,h)(g_h)$ satisfies the LDP and the rate function equals
    $$a \mapsto \inf_{\substack{\Q \in \Pt \\ \Q^{r-\neighb}(t,h)(g_h) - \Q^{r'-\neighb}(t,h)(g_h) = a}}\hspace{-2cm}H(\Q).$$
    In particular, \eqref{expAppEq} gives that
    $$-\hspace{-2cm}\inf_{\substack{\Q \in \Pt \\ |\Q^{r-\neighb}(t,h)(g_h) - \Q^{r'-\neighb}(t,h)(g_h)| > \varepsilon}}\hspace{-2cm}H(\Q) < -\alpha,$$
    so that the asserted upper bound \eqref{r0PropEq} holds for every $\Q \in \Pt$ with $H(\Q) \le \alpha$.
   \end{proof}

\subsection{Proof of Proposition \ref{tvProp}}
\label{expAppSec}

To prove exponentially good approximation, we compare the S-PAM with the Poisson random connection model (RCM) \cite{netCom}. In general, the S-PAM differs from the RCM substantially because preferential attachment leads to high-degree nodes. However, checking whether the $h$-neighbourhood of a given vertex is of a certain form entails a uniform bound on the maximum size of the relevant in-degrees, so that any discrepancy between $G_n$ and $G_n^r$ must come from an edge of length at least $r$ in the RCM. The integrability of the profile function implies that this is a rare event.

Carrying out this program rigorously involves several intermediate steps that we state now and prove later in this section.  First, the proportion of nodes arriving at early times is asymptotically negligible.
\begin{lemma}[Early vertices]
	\label{earlyLem}
	Let $\e>0$. Then,
	$$\limsup_{\sigma \to 0}\limsup_{n \to \infty}\frac1n \log \P(\#(X \cap (\T_n \times [0, \sigma])) > \e n)  = -\infty.$$
\end{lemma}

%
%
Second, contributions from neighbourhoods around vertices located in highly dense regions of $\T_n$ can also be ignored. To simplify the presentation, we assume in the following that $n' = n^{1/d}$ is an integer. Now, we partition $\T_n$ into cubes $Q_z = z + [-1/2, 1/2]^d$ centred at sites of discrete torus $\Z^d/n'$. Moreover, we let $N_z = \#(X \cap (Q_z \times [0,1]))$ denote the number of vertices in $Q_z$. 
    For a threshold $m\ge1$, we define $z \in \Z^d/n'$ to be \emph{$m$-dense}, in symbols $z \in \Sd$, if $N_z \ge m$ and \emph{$m$-sparse} otherwise. 
\begin{lemma}[High-density regions]
	\label{denseLem}
	Let $\la >0$. Then,
	$$\lim_{m \to \infty}\sup_{n\ge1} \frac1n \log\E\Big[\exp\Big(\la \sum_{z \in \Sd} N_z\Big)\Big] = 0.$$
\end{lemma}

Third, we deduce the finiteness of an exponential moment, thereby helping to bound the number of edges emanating from a given vertex. Fix $b, m>0$ and let $\{N'_z\}_{z \in \Z^d}$ be a family of independent random variables, where $N'_z$ follows the distribution of a binomial random variable with $m$ trials and success probability $\varphi(b(|z| - \sqrt d)_+)$.
\begin{lemma}[Exponential moment]
	\label{finExpmomLem}
	Let $b, m, \la > 0$. Then,
	$$\E\Big[\exp\Big(\la \sum_{z \in \Z^d} N'_z\Big)\Big] <\infty.$$
\end{lemma}

Finally, we bound the number of long edges ending in an $m$-sparse cube. For $x \in X$ we define $z(x) \in \Z^d/n'$ to be such that $x \in Q_{z(x)}$. In words, $z(x)$ is the centre of the cube containing $x$.
\begin{lemma}[Long edges]
	\label{longEdgeLem}
	Let $b, m, \la > 0$. Then,
	$$\lim_{r \to \infty}\sup_{n\ge1} \frac1n \log\E\Big[\exp\Big(\la\hspace{-0.1cm} \sum_{\substack{x,y \in X \\ |x - y| > r}}\one\{z(y) \not \in \Sd \text{ and }  V_{x,y} \le \varphi(b|x-y|)\}\Big)\Big] = 0.$$
\end{lemma}

Before establishing Lemmas \ref{earlyLem}--\ref{longEdgeLem}, we show how they enter the proof of the main result.

\begin{proof}[Proof of Proposition \ref{tvProp}]
	Without loss of generality, set $t=1$.  We derive a bound for the number of \emph{bad vertices $x \in X$}, i.e.~vertices whose $h$-neighbourhood is isomorphic to $g_h$ in $G_n$ but not in $G_n^r$. The corresponding bound with interchanged roles of $G_n$ and $G_n^r$ follows from similar arguments. Hence, for any such $x$ there exists a vertex $x'$ in the $(h-1)$-neighbourhood of $x$ in $G_n^r$ and a vertex $y'$ with $|y' - x'| > r$ such that $y'$ connects to $x'$ in $G_n$ but not in $G_n^r$. By Lemma \ref{earlyLem}, we may restrict our attention to vertices born after time $\sigma$. 
	In particular, writing $\ell$ for the number of vertices in $g_h$, we see that $y'$ is connected to $x$ in the RCM $\Grak$ with vertex set $X$ and where there is an edge from $(x_1, s_1)$ to $(x_2, s_2)$ if and only if $V_{x_1,x_2} \le \varphi_*(|x_1 - x_2|)$, where $\varphi_*(\rho) = \varphi(\sigma\rho/f(\ell))$.

	To obtain bounds on exponential moments, we aim to restrict our attention to neighbourhoods intersecting only $m$-sparse cubes. More precisely, a self-avoiding path $\pi = ((x_0, s_0),  \ldots, (x_j, s_j))$ in $\Grak$ is \emph{$m$-sparse } if 
	\begin{enumerate}
		\item $|x_i - x_{i+1}| \le r$ and $s_{i+1} \le s_i$ for every $i\ge0$,
		\item $z(x_i) \not \in \Sd$ for every $i\ge1$.
	\end{enumerate}
	It is on purpose, that we do not impose $z(x_0) \not \in \Sd$.  With this definition, every bad vertex is contained in an $m$-sparse connected path that either starts at a vertex contained in an $m$-dense cube or features an in-going edge of length at least $r$ in the RCM $\Grak$. Writing $\C_1$ and $\C_2$ for the families of vertices contained in these types of paths, it suffices to provide upper bounds for $\#\C_1$ and $\#\C_2$. More precisely, by the exponential Markov inequality it suffices to show that for every $\la>0$,
	\begin{align}
		\label{mApproxEq}
		\limsup_{m \to \infty}\limsup_{n \to \infty} \frac1n \log\E[\exp(\la \#\C_1)] = 0,
	\end{align}
	and
	\begin{align}
		\label{rApproxEq}
		\limsup_{r \to \infty}\limsup_{n \to \infty} \frac1n \log\E[\exp(\la \#\C_2)] = 0.
	\end{align}

	To begin with, we show \eqref{mApproxEq} and introduce the \emph{$m$-sparse connected component $C_x$ at $x \in X$} as the union of all vertices of $m$-sparse paths in $\Grak$ starting at $x \in X$ and consisting of at most $\ell$ hops. First, recalling the definition of $N'_z$ from the paragraph preceding Lemma \ref{finExpmomLem}, the number of edges in $\Grak$ from any $x' \in X$ to a vertex in $z \not \in \Sd$ is stochastically dominated by $N'_{z(x') - z}$.  In particular, the independence of the family $\{V_{x,y}\}_{x,y \in X}$ implies that conditioned on $X$, the size of $C_x$ is stochastically dominated by the offspring until generation $\ell$ of a Galton-Watson process with offspring distribution $N' = \sum_{z \in \Z^d}N'_z$.  Despite the independence assumption on the collection $\{V_{x, y}\}_{x, y \in X}$, the $m$-sparse connected components $C_x$ at different points $x \in X$ are not independent because they can share common vertices. Nevertheless, the size of their union is stochastically dominated by the sum of the component sizes, when each component is explored independently, see \cite[Lemma 2.3]{DePi96}. Hence, noting that Lemma \ref{finExpmomLem} yields the finiteness of the cumulant generating function $c(\la) = \log\E[\exp(\la N')]$, we arrive at 
	\begin{align}
		\label{densemomEq}
		\frac1n \log\E[\exp(\la \#\C_1)] \le \frac1n \log\E[\exp(c^{(\ell)}(\la) \#\{x \in X:\, \text{$z(x) \in \Sd$}\})],
	\end{align}
	$c^{(\ell)}$ denotes the $\ell$-fold iteration of the function $c$. By Lemma \ref{denseLem}, the right-hand side tends to 0 as $m \to \infty$. 

	In order to show \eqref{rApproxEq}, we proceed in precisely the same way until arriving at the analog of \eqref{densemomEq}, where the number of vertices contained in an $m$-dense cube is replaced by the number of vertices that are contained in an $m$-sparse cube and are the endpoint of an edge in $\Grak$ of length at least $r$. Instead of Lemma \ref{denseLem}, now Lemma \ref{longEdgeLem} implies that the resulting expression tends to 0 as $r \to \infty$.
\end{proof}

%
%
Finally, we prove Lemmas \ref{earlyLem}--\ref{longEdgeLem}.

\begin{proof}[Proof of Lemma \ref{earlyLem}]
    For fixed $\sigma > 1$, the number $\#(X \cap (\T_n \times [0, \sigma]))$ of Poisson points born before time $\sigma$ is a Poisson random variable with parameter $n\sigma$. In particular, by the Poisson concentration inequality \cite[Lemma 1.2]{penrose},
	$$\frac1n \log \P(\#(X \cap (\T_n \times [0, \sigma])) > \e n) \le -\frac{\e}{2}\log(\e \sigma^{-1}),$$
    where the right-hand side tends to $-\infty$ as $\sigma \to 0$.
\end{proof}

\begin{proof}[Proof of Lemma \ref{denseLem}]
	First, by the independence property of the Poisson point process,
	$$\frac1n \log\E\Big[\exp\Big(\la \sum_{z \in \Sd} N_z\Big)\Big] = \log\E[\exp(\la \one\{o \in \Sd\} N_o)].$$ 
	Since $N_o$ has all exponential moments, the claim follows from the monotone convergence theorem.
\end{proof}

\begin{proof}[Proof of Lemma \ref{finExpmomLem}]
	By independence of the $\{N'_z\}_{z \in \Z^d}$, we can expand the exponential moment as 
	\begin{align*}
	\log\E\Big[\exp\Big(\sum_{z \in \Z^d}N'_z\Big)\Big] &= m\sum_{z \in \Z^d} \log\big(1 + \varphi(b(|z| - \sqrt d)_+)(e^\la - 1)\big)\\
						    &\le m(e^\la - 1)\sum_{z \in \Z^d} \varphi(b(|z| - \sqrt d)_+).
	\end{align*}
	Now, the infinite series in the last line converges because	$\varphi$ is integrable and decreasing.
\end{proof}

\begin{proof}[Proof of Lemma \ref{longEdgeLem}]
	Since the $\{V_{x,y}\}_{x, y \in X}$ are iid, we see that conditioned on $X$ the events
	$$\{V_{x,y} \le \varphi(b|x-y|)\}$$
	are independent for different values of $x$ and $y$. Hence, the sum in the exponential is stochastically dominated by
	$$\sum_{{x \in X}}\sum_{z':\,|z(x) - z'| > r/2}N_{x, z'},$$
	where the $N_{x, z'}$ are independent binomial random variables with $m$ trials and success probability $\varphi(b(|z(x) - z'| -  \sqrt d))$. In particular, by the formula for the characteristic function of compound Poisson sums,
	\begin{align*}
		\frac1n\log\E\Big[\exp\Big( \la\sum_{{x \in X}}\sum_{{z':\,|x - z'| > r/2}} N_{x, z}\Big)\Big] &= \E\Big[\exp\Big(\la \sum_{{z:\,|z| > r/2}}N_{o, z}\Big)\Big] - 1 \\
													   &= \exp\Big(m\sum_{{z:\,|z| > r/2}}\log(1+\varphi(b(|z| -  \sqrt d)) (e^\la - 1))\Big) - 1\\
													   & \le \exp\Big(m (e^\la - 1) \sum_{{z:\,|z| > r/2}}\varphi(b(|z| -  \sqrt d))\Big) - 1,
	\end{align*}
where the last sum converges by the integrability assumption on the profile function. Sending $r \to \infty$ concludes the proof.
\end{proof}

\subsection{Proof of Corollary \ref{ldpCor}}
Since Theorem \ref{ldpThm} already provides an LDP for fixed times, the proof of Corollary \ref{ldpCor} reduces to verifying exponential tightness in the Skorohod topology.

\begin{proof}[Proof of Corollary \ref{ldpCor}]
    By Corollary \ref{ldpThm}, the rescaled number of nodes $L_n^{\deg; \le k}$ of degree at most $k$  satisfies an LDP in the product topology. In particular, by \cite[Corollary 4.2.6]{dz98}, it suffices to establish exponential tightness in the Skorohod topology. To this end, we use a criterion established in \cite[Theorem 4.1]{FeKu06}:
    \begin{enumerate}
        \item $L_n^{\deg; \le k}(t)$ is exponentially tight for every $t \le 1$, and
        \item $\limsup_{\eta \to 0}\limsup_{n \to \infty}\frac1n \log \P(w_\eta'(L_n^{\deg; \le k}) > \e) = -\infty,$
    \end{enumerate}
    where 
    $$w_\eta'(L_n^{\deg; \le k}) = \inf_{0 = t_0 < \cdots < t_j = 1:\, \min_{ i \le k}|t_i - t_{i-1}| > \eta} \max_{ i \le j} \sup_{s, t \in [t_{i-1}, t_i)}|L_n^{\deg; \le k}(s) - L_n^{\deg; \le k}(t)|$$ 
    denotes the Skorohod-adapted modulus of continuity. 

	Exponential tightness of $L_n^{\deg; \le k}(t) \le n^{-1}\#X$ is a consequence of the exponential tightness of the rescaled Poisson random variable $n^{-1}\#X$.  Therefore, we concentrate on item (2). Here, we proceed along the lines of the proof of Proposition \ref{tvProp}. Fixing an interval $I \subset [0,1]$ of length $\eta$, we distinguish two cases. First, assume that $I \subset [0, 2\sigma]$ for some small $\sigma > 0$. Then, during the time interval $I$, the truncated in-degree of each node can change by at most $k \ge 1$. Therefore, 
    $$\frac1n \log \P(w_\eta'(L_n^{\deg; \le k}) > \e) \le n^{-1} \log \P(k\#(X \cap (\T_n \times [0, \sigma])) > n\e),$$
    which by Lemma \ref{earlyLem} tends to $-\infty$ as $\sigma \to 0$.

    Hence, from now on we may assume that $I \subset [\sigma, 1]$. Then, setting again $\varphi_*(\rho) = \varphi(\sigma\rho/f(\ell))$, we proceed  similarly to Proposition \ref{tvProp} and introduce the quantity
	$$N = \#\big\{x \in X:\, V_{x, y} \le \varphi_*({|x - y|})\text{ and }Z_x(t-) \le k \text{ for some }(y, t) \in X \cap (\T_N \times I)\big\}.$$
    Since the total number of in-degree changes during the time interval $I$ is at most $k N$, it suffices to derive a suitable upper bound for the latter. In particular, $N \le N^{(\ms s)} + \sum_{z \in \Sd} N_z$, where
    $$N^{(\ms s)} = \sum_{\substack{x, y \in X\\ z(x) \not \in \Sd}} \one\big\{(y,t) \in X \cap (\T_N \times I) \text{ and }V_{x, y} \le  \varphi_*({|x - y|})\big\}.$$
    Using the exponential Markov inequality and Lemma \ref{denseLem}, it suffices to show that for any fixed $b>0$ we have that 
    $$ \log \E[\exp(b N^{(\ms s)})] \le 2n$$
    if $\eta$ is chosen sufficiently small. To achieve this goal, we note that as in the proof of Lemma \ref{longEdgeLem} the random variable $N^{(\ms s)}$ is stochastically dominated by
    $$\sum_{{(y,t) \in X \cap (\T_N \times I)}}\sum_{z \in \Z^d/n'}N_{y, z},$$
    where again the $N_{y, z}$ are independent binomial random variables with $m$ trials and success probability $\varphi(b(|z| - \sqrt d)_+)$. Since $X \cap (\T_N \times I)$ is a Poisson point process with intensity $\eta n$, this time we arrive at 
    $$\frac1n \log \E[\exp(b N^{(\ms s)})] \le  \eta \Big(\E\Big[\exp\Big(b \sum_{z \in \Z^d} N'_z\Big)\Big] - 1\Big),$$
    which by Lemma \ref{finExpmomLem} tends to 0 as $\eta \to 0$.
\end{proof}

\section{Extensions and open problems}
\label{conclusion}

\subsection{Distances}
It is straightforward to see, by checking the error terms in the proofs of Propositions \ref{prop:layersB} and  \ref{prop:layersC} that Theorem \ref{distanceThm} remains valid for the slightly more general choices of attachment and profile function studied in \cite{jacMor2}, thus the following extension of our distance result holds.
\begin{corollary}
Let the preferential attachment function $f$ satisfy
\begin{equation*}
\lim_{k \to \infty}\frac{f(k)}k=\gamma\in(0,1)
\end{equation*}
and the profile function $\varphi$
\begin{equation*}
\varphi(x)=L(x)x^{-\delta},
\end{equation*}
for some $\delta\in(1,\infty)$ and some slowly varying function $L:(0,\infty) \to  (0,\infty)$, then the conclusion of Theorem \ref{distanceThm} remains valid.
\end{corollary}

We also conjecture that the upper bound on the distances in the giant component is sharp.
\begin{conjecture}[Lower bound on distances for $\gamma > \nicefrac\delta{(1 + \delta)}$]
Let $Y, Y'$ be chosen uniformly among the vertices of $C_n$. It holds that \[
\textup{dist}_n(Y, Y')=\big(4 + o(1)\big)\frac{\log\log n}{\log \frac {\gamma}{\delta(1-\gamma)}} \text{ with high probability as }n \to \infty.
\]
\end{conjecture}
The presence of positive correlations between the edges induced by the geometry of the model prevents the use of a simple first moment method to show that shorter paths do not exist. Since the machinery to at least partially overcome this problem has already been developed in \cite{jacMor2} to show non-robustness of the S-PAM in a certain parameter range, we are confident that a matching lower bound can be obtained. We have, however, so far not been able to match the calculation of the upper bounds presented above. It is an interesting open problem to complement our result by the lower bound.\\

A related question is that of distances in the finite components and in the non-robust regime. By comparison to other continuum percolation models one would expect that when the preferential attachment is very weak, e.g.~sub-linear, then the graph distances are not shorter than polynomial in the component size. Of course, it is not even known under which precise conditions the model is sub-critical. The bounds known are proved in \cite{jacMor2} and provide non-robustness: $\gamma<\nicefrac{\delta-1}\delta$ in dimension $1$ and $\gamma<\nicefrac12$ in arbitrary dimension. No explicit parameter values in terms of profile and preferential attachment function are known for which there is a priori no giant component. It is also quite likely that there is a regime featuring a non-robust giant component akin to a supercritical continuum percolation cluster with additional long edges, thereby reducing the length of shortest paths drastically to $(\log n)^a$ for some $a>1$ or even to $\log n$, which would correspond to non-spatial models.

\subsection{Large deviations}
We stress that Corollary \ref{ldpThm} holds under the assumption that $\M(\Zn)$ carries the vague topology, thereby falling short of capturing observables depending on arbitrarily high in-degrees. This resembles the situation in \cite[Theorem 1.4]{choi}, which relies on the countable product space $C([0,1]; \R)^{\Z_{\ge0}}$ representing the evolutions of proportions of nodes with given in-degrees. Since this space carries the product topology, effectively also \cite[Theorem 1.4]{choi} is not adapted to dealing with arbitrarily high in-degrees. In \cite{subLinPA}, this constraint does not appear since the dynamics of the PAM in the sub-linear regime differs decisively from the one in the linear regime.

We conjecture that Corollary \ref{ldpThm} continues to hold in the total-variation topology. This would follow from Conjecture \ref{ldpCon} below, which would also open the door to a large deviation analysis of more refined network characteristics such as the clustering coefficient.

\begin{conjecture}
    \label{ldpCon}
    The normalised number of points of in-degree larger than $k$ is an exponentially good approximation of 0. That is, for every $\varepsilon > 0$,
    $$ \limsup_{k  \to  \infty}\limsup_{n  \to  \infty}\frac1n\log\P(\#\{x \in X:\, Z_x(1) \ge k\} > n\varepsilon) = -\infty.$$
\end{conjecture}

\appendix    
\section{Appendix}
\label{appSec}
Here, we provide statements from the literature and auxiliary calculations used in the main text. We start by a refined version of Lemma \ref{lem:2conn1}, where we write $\kappa_d = |B_1(o)|$ for the volume of the unit ball in $\R^d$.
\begin{lemma}\label{APP1}
	Denote by $G^\lambda_n$ the S-PAM built from a homogeneous Poisson point process $X^\lambda$ with intensity $\lambda > 0$. Let $(x,s), (y,t) \in X^\lambda$ be two vertices satisfying $s,t \le \nicefrac12$. Moreover, let $Z_x(\nicefrac12)$ and $Z_y(\nicefrac12)$ denote their respective in-degrees in $G^\lambda_n$ at time $\nicefrac12$ and define \[k(x,y)=f\big(Z_x(\nicefrac12)\big)\varphi\Big(\frac{\big(f\big(Z_x(\nicefrac12)\big)^{\nicefrac1d}+ |x - y|\big)^d}{f\big(Z_y(\nicefrac12)\big)}\Big),\] and \[
	Q(x,y)=\frac{\varphi(1)\kappa_d}2 \big(k(x,y)\vee k(y,x)\big).
\]
	Then, conditional on $X^\lambda \cap (\T_n \times [0, 1/2])$,	$x$ and $y$ are $2$-connected in $G^\lambda_n$ by using only vertices from $X^\lambda\cap (\T_n \times [1/2, 1])$ with probability exceeding $1-\mathrm e^{-\lambda Q(x,y)}$.
\end{lemma}
\begin{proof}
	Set $z_x=Z_x(\nicefrac12)$, $z_y=Z_y(\nicefrac12)$ and let $X^\circ$ denote those vertices $(w,r)$ of $X^{\lambda}$ which lie in $B_{f(z_x)^{1/d}}(y) \times [1/2, 1]$ and satisfy $V_{x, w} \le \varphi(1)$ and $V_{y,w} \le \varphi(r |y - w|^d / f(z_y))$. In particular, all $(w,r)\in X^\circ$ are 2-connectors. By the restriction theorem \cite[Theorem 5.2]{poisBook}, $X^\circ$ forms a Poisson point process with intensity 
	$$\int_{B_{f(z_x)^{1/d}}(y)} \frac{\lambda \varphi(1)}2 \varphi(r |y - w|^d / f(z_y))  \d w \ge \frac{\lambda\varphi(1)\kappa_d f(z_x)}2\varphi\Big(\frac{ (f(z_x)^{\nicefrac1d} + |x - y|)^d}{f(z_y)}\Big).$$
	In particular, 
	$$\P(X^\circ = \emptyset) \le \exp(-\tfrac12 {\lambda\varphi(1)\kappa_d } k(x,y)),$$ 
	so that reversing the roles of $x$ and $y$ yields the assertion of the lemma.
\end{proof}

The next statement ensures that old vertices tend to be good. Let $Z^n(s,\cdot)$ denote the generic in-degree evolution of a vertex born at time $s$ in $G_n$, noting that its spatial position has no influence on $Z^n(s,\cdot)$. Assume that $G_n$ is built from a Poisson process of intensity $\lambda>0$
\begin{lemma}[{\cite[Lemma 24]{jacMor2}}]\label{lem:oldgood}
Let $(x,s)\in G_n$ be born at time $s \le \nicefrac12$. There exists a function $g = g_{\lambda}$ decaying faster than any power at $\infty$ such that
\[
	\lim_{s \to 0}\sup_{\substack{n \ge 1\\ n \log n \geq s^{-1}}} \P_{(x, s)}\big(Z^n(s,\nicefrac12) \le s^{-\gamma}/g(s^{-1})\big) \overset{s \to 0}{\longrightarrow} 0.
\]
Consequently, 
\[
	\lim_{s \to 0}\sup_{\substack{n \ge 1\\ n \log n \geq s^{-1}}} \P_{(x, s)}\big((x,s) \text{ is not good}\big) \overset{s \to 0}{\longrightarrow} 0.
\]
\end{lemma}
\begin{proof}
See \cite[p.~1720]{jacMor2} for the proof with $\lambda=1$. A higher intensity increases the degree of $(x,s)$. That lowering the intensity makes no difference to the sub-polynomial decay can be seen easily, since $g$ may be replaced by any other increasing function of sufficiently slow decay, cf.~the proofs of Lemma 23 and 24,\cite[p.~1720]{jacMor2}. In particular, reducing the intensity of the Poisson process can be compensated for by increasing $g$ by a constant factor. This has no influence on its sub-polynomial decay.
\end{proof}

The following corollary is obtained directly from the proof of {\cite[Lemma 24]{jacMor2}}: for a given birth time $s$, using a scaling property of the degree evolutions, it is actually sufficient to consider connections to $(x,s)$ in an $s^{-1/d}$ environment of $x$. This fact is also used, without explicit mentioning, in the proof of \cite[Proposition 13]{jacMor2}.
\begin{corollary}\label{lem:localgood}
For any $x \in \T_n$ we have
\[
	\inf_{\substack{s < 1/2\\ n\ge 1 }} \P_{(x, s)}((x, s) \text{ is locally good})>0.
\]
\end{corollary}
The following short calculation shows that the error in the proof of Proposition \ref{prop:layersB} can be made arbitrarily small.
\begin{lemma}\label{lem:vanish}
For any $q,\varepsilon>0$ and $\alpha>1$ we have 
	\[\lim_{s \to 0} \sum_{k \ge 1} \exp\big(-qs^{-\varepsilon \alpha^k}\big)=0.
\]
\end{lemma}
\begin{proof}
	Clearly $r(s)=\exp(-qs^{-\varepsilon}) \to 0$ as $s \to 0$, and

	\[\sum_{k \ge 1}r\big(s^{\alpha^k}\big) \le \sum_{k:\, \alpha^k \le k }r(s^{\alpha^k}) + \frac1{r(s)}-1,\]
which vanishes as $r(s) \to 0.$
\end{proof}

\bibliography{wias}
\bibliographystyle{abbrv}

\end{document}